\documentclass[12pt,reqno]{amsart}

\usepackage{amssymb,amsthm}
\usepackage{amsfonts,cmtiup,comment,stmaryrd}

\makeatletter
\def\blfootnote{\xdef\@thefnmark{}\@footnotetext}
\makeatother

\newtheorem{theorem}{Theorem}[section]
\newtheorem{question}[theorem]{Question}
\newtheorem{lemma}[theorem]{Lemma}
\newtheorem{proposition}[theorem]{Proposition}
\newtheorem{corollary}[theorem]{Corollary}

\newtheorem{hyp}[theorem]{Hypothesis}
\newtheorem{conj}[theorem]{Conjecture}

\theoremstyle{definition}

\newtheorem{remark}[theorem]{Remark}
\newtheorem*{definition*}{Definition}

\newcommand{\bt}{\begin{theorem}}
\newcommand{\et}{\end{theorem}}
\newcommand{\bc}{\begin{corollary}}
\newcommand{\ec}{\end{corollary}}
\newcommand{\bpr}{\begin{proposition}}
\newcommand{\epr}{\end{proposition}}
\newcommand{\be}{\begin{equation}}
\newcommand{\ee}{\end{equation}}
\newcommand{\bp}{\begin{proof}}
\newcommand{\ep}{\end{proof}}
\newcommand{\bconj}{\begin{conj}}
\newcommand{\econj}{\end{conj}}
\newcommand{\bl}{\begin{lemma}}
\newcommand{\el}{\end{lemma}}
\newcommand{\bh}{\begin{hyp}}
\newcommand{\eh}{\end{hyp}}
\newcommand{\br}{\begin{remark}}
\newcommand{\er}{\end{remark}}

\let\leq=\leqslant
\let\geq=\geqslant
\textwidth=16cm \textheight=24cm \numberwithin{equation}{section}
\newcommand{\ed}{\end{document}}
\begin{document}

\title[Engel-type subgroups]{Engel-type subgroups and length parameters  of finite groups}

\author{E. I. Khukhro}

\address{School of Mathematics and Physics, University of Lincoln, Lincoln, LN6 7TS, U.K., and\\ Sobolev Institute of Mathematics, Novosibirsk, 630090, Russia.}
 \email{khukhro@yahoo.co.uk}

\author{P. Shumyatsky}

\address{Department of Mathematics, University of Brasilia, Brasilia, DF~70910-900, Brazil}

\email{pavel@unb.br}


\begin{abstract}
Let $g$ be an element of a finite group $G$. For a positive integer $n$, let $E_n(g)$ be the subgroup generated by all commutators  $[...[[x,g],g],\dots ,g]$ over $x\in G$,  where $g$ is repeated $n$ times. By Baer's theorem, if $E_n(g)=1$, then  $g$ belongs to the Fitting subgroup $F(G)$. We generalize this theorem in terms of certain  length parameters of $E_n(g)$.  For soluble $G$ we prove that if, for some $n$, the Fitting height of $E_n(g)$ is equal to $k$, then $g$ belongs to the $(k+1)$th Fitting subgroup $F_{k+1}(G)$. For nonsoluble $G$ the results are in terms of  nonsoluble length and generalized Fitting height. The generalized Fitting height $h^*(H)$ of a finite group $H$ is the least number $h$ such that $F^*_h(H)=H$, where  $F^*_0(H)=1$, and $F^*_{i+1}(H)$ is the inverse image of the generalized Fitting subgroup $F^*(H/F^*_{i}(H))$. Let $m$ be the number  of prime factors of $|g|$ counting multiplicities. It is proved that if, for some $n$, the  generalized Fitting height of $E_n(g)$ is equal to $k$, then $g$ belongs to $F^*_{f(k,m)}(G)$, where $f(k,m)$ depends only on $k$ and $m$. The nonsoluble length~$\lambda (H)$ of a finite group~$H$ is defined as the minimum number of nonsoluble factors in a normal series each of whose factors either is soluble or is a direct product of nonabelian simple groups. It is proved that if  $\lambda (E_n(g))=k$, then $g$ belongs to a normal subgroup whose nonsoluble length is bounded in terms of $k$ and $m$. We also state conjectures of stronger results independent of $m$ and show that these conjectures reduce to a certain question about automorphisms of direct products of finite simple groups.
\end{abstract}

\maketitle

\blfootnote{The first
author was supported  by the Russian Science Foundation, project no. 14-21-00065,
and the second
by the Conselho Nacional de Desenvolvimento Cient\'{\i}fico e Tecnol\'ogico (CNPq), Brazil. The first author thanks  CNPq-Brazil and the University of Brasilia for support and hospitality that he enjoyed during his visits to Brasilia.}

\section{Introduction}
By the well-known theorem of Baer \cite[Satz~III.6.15]{hup}, an element $g$
of a finite group $G$ belongs to the Fitting subgroup $F(G)$ if and only
if it is a left-Engel element, that is, if $[x,g,g,\dots , g]=1$ for all $x\in G$,
where $g$ is repeated in the commutator  sufficiently many times.
(Throughout the paper, we use the left-normed simple commutator notation
$[a_1,a_2,a_3,\dots ,a_r]=[...[[a_1,a_2],a_3],\dots ,a_r]$.) In this paper
we generalize this result using the subgroups
$$
E_n(g)=\langle [x,\underbrace{g,\dots ,g}_{n}]\mid x\in G\rangle.
$$

 Recall that the Fitting series is defined  starting from  $F_0(G)=1$, and then by induction $F_{k+1}(G)$ is  the inverse image of $F(G/F_k(G))$. Our first result is about soluble groups.

\begin{theorem}\label{t-sol}
Let $g$ be an element of a finite soluble group $G$, and  $n$ a positive integer.
If  the Fitting height of $E_n(g)$ is equal to $k$, then $g$ belongs to  $ F_{k+1}(G)$.
\end{theorem}

For nonsoluble finite groups we prove similar results in terms of the nonsoluble length and generalized Fitting height of $E_n(g)$.
We recall the relevant definitions. The generalized Fitting subgroup $F^*(G)$ is the product of the Fitting subgroup $F(G)$ and all subnormal quasisimple subgroups (here a group is quasisimple if it is equal to its derived subgroup and its  quotient by the centre is a non-abelian simple group). Then the \textit{generalized Fitting series} of $G$ is defined starting from  $F^*_0(G)=1$, and then by induction $F^*_{i+1}(G)$ is the inverse image of $F^*(G/F^*_{i}(G))$. The least number $h$ such that $F^*_h(G)=G$ is naturally defined to be  the \textit{generalized Fitting height} $h^*(G) $ of $G$. Clearly, if $G$ is soluble, then $h^*(G)=h(G)$ is the ordinary Fitting height of $G$.

\begin{theorem}\label{t2}
Let $m$ and $n$ be  positive integers, and let $g$ be an element of a finite group $G$ whose order $|g|$ is equal to the product of $m$ primes counting multiplicities.  If the generalized Fitting height of $E_n(g)$ is equal to $k$, then $g$ belongs to  $ F^*_{j}(G)$ for some $j\leq ((k+1)m(m+1)+2)(k+3)/2$.
\end{theorem}

This theorem follows from Theorem~\ref{t-sol} on soluble groups and Theorem~\ref{t3} below concerning another length parameter. Namely,  the  \textit{nonsoluble length} $\lambda (G)$  of a finite group $G$ is defined as the minimum number of nonsoluble factors in a normal series each of whose factors either is soluble or is a direct product of nonabelian simple groups. (In particular, the group is soluble if and only if its nonsoluble length is $0$.) Bounds for the nonsoluble length or/and generalized Fitting height of a finite group $G$ greatly facilitate using the classification of finite simple groups (and are  themselves often obtained by using the classification). Most notably such bounds were used in the reduction of the Restricted Burnside Problem to soluble and nilpotent groups in  the Hall--Higman paper \cite{ha-hi}.  Such bounds  also find applications in the study of profinite groups. Examples include Wilson's reduction of the problem of local finiteness of periodic profinite groups to pro-$p$ groups in \cite{wil83} and our recent paper  \cite{khu-shu} on similar problems. (Both the Restricted Burnside Problem and the problem of local finiteness of periodic profinite groups were solved by Zelmanov \cite{zel90, zel91, zel92}.)

Similarly to the generalized Fitting series, we can define terms of the `upper nonsoluble series': let  $R_i(G)$ be the maximal normal subgroup of $G$ that has nonsoluble length $i$ (so that, in particular, $R_0(G)$ is the soluble radical of $G$).

\begin{theorem}\label{t3}
Let $m$ and $n$ be  positive integers, and let $g$ be an element of a finite group $G$ whose order $|g|$ is equal to the product of $m$ primes counting multiplicities.  If the nonsoluble length  of $E_n(g)$ is equal to $k$,
then $g$ belongs to  $ R_j(G)$ for some $j\leq {(k+1)m(m+1)/2}$.
\end{theorem}

Note that $E_n(g)$ is \textit{not} a subgroup of the type $[...[[G,g],g],\dots ,g]$  formed by taking successive commutator subgroups, which is subnormal.
But $E_n(g)$ is not subnormal in general.

Our results on nonsoluble groups depend on the classification of finite simple groups in so far as the validity of the Schreier
conjecture on solubility of the group of outer automorphisms of a finite simple group.
In Theorems~\ref{t2} and \ref{t3}, nonsoluble length and generalized Fitting height have  bounds depending both on the parameters of $E_n(g)$ and on the number of prime factors in the order of $g$. We conjecture that stronger results may hold, not depending on the order of $g$. Namely, we conjecture that if, for an element $g$ of a finite group $G$
the nonsoluble length of $E_n(g)$ is equal to $k$, then $g$
belongs to  $ R_{k}(G)$. We also conjecture that if  the generalized Fitting height of $E_n(g)$ is equal to $k$, then $g$
belongs to  $ F^*_{k+1}(G)$.
These conjectures can be reduced to certain questions about automorphisms of direct products of nonabelian simple groups.
The questions, though, seem rather difficult and remain open. We discuss them in \S\,\ref{s-rem}.

The authors thank Professor Robert M. Guralnick for useful discussions.

\section{Preliminaries and proof of Theorem~\ref{t-sol}}

We notice  nice `radical' properties of the subgroups $F_i(G)$, $F^*_i(G)$, and $R_i(G)$, which will be often used without special references. Namely,  it follows from the definitions that if $N$ is a   normal subgroup of $G$, then
$$
F_i (N)=N\cap  F_i(G), \qquad F^*_i (N)=N\cap   F^*_i(G), \qquad R_i (N)=N\cap  R_i(G).
$$

For the Fitting subgroups we also have the inclusions $F_i(G)\cap H\leq F_i(H)$ for any, not necessarily normal, subgroup $H\leq G$. However, similar inclusions do not hold in general for the subgroups $F_i^*(G)$ and $R_i(G)$.

As a consequence we have the following.

\begin{lemma}
\label{l-sn}
Let $N$ be a subnormal subgroup of a finite group $G$. Then

{\rm (a)} the Fitting height (when $G$ is soluble), the generalized Fitting height, and the nonsoluble length of $N$ do not exceed the corresponding parameters of $G$, and

{\rm (b)} the Fitting height (when $N$ is soluble), the generalized Fitting height, and the nonsoluble length of the normal closure $\langle N^G\rangle$ are equal to the corresponding parameters of~$N$.
\end{lemma}

It also follows from the definitions that
$$
F_i (G/N)\geq  F_i(G)N/N, \quad F^*_i (G/N)\geq  F^*_i(G)N/N, \quad R_i (G/N)\geq  R_i(G)N/N
$$
for any normal subgroup $N$ of $G$.

When we consider a group $A$ acting by automorphisms on a  group $G$, we regard $A$ as a subgroup of the natural semidirect product $GA$, so that we can form the mutual  commutator subgroup $[G,A]$ and use the centralizer notation for the fixed-point subgroup $C_G(A)$.
Throughout the paper we use without special references the well-known
properties of coprime actions:  if $\alpha $ is an automorphism of a finite
group $G$ of coprime order, $(|\alpha |,|G|)=1$, then $ C_{G/N}(\alpha )=C_G(\alpha )N/N$ for any $\alpha $-invariant normal subgroup $N$, the equality  $[G,\alpha ]=[[G,\alpha,],\alpha ]$ holds, and if $G$ is in addition abelian, then $G=[G,\alpha ]\times C_G(\alpha )$.

We generalize slightly the notation introduced above.  If $H$ is a $g$-invariant subgroup of a group $G$, where $g\in G$ or $g\in {\rm Aut}\,G$, then
let
$$
E_{H,n}(g)=\langle [h,\underbrace{g,g,\dots, g}_{n}]\mid h\in H\rangle.
$$
Thus, $E_n(g)=E_{G,n}(g)$ when $g\in G$ and it is clear from the context which group $G$ is involved. It is clear that $C_G(g)$ normalizes  $E_{G,n}(g)$.

We are now going to prove Theorem~\ref{t-sol}.
The proof reduces to the following proposition.

\begin{proposition}
\label{pr-sol}
Let $\alpha $ be an automorphism of a finite soluble group $G$ such that $G=[G,\alpha]$.  Let $n$ be a positive integer. Then $E_{G,n}(\alpha ) =G$.
\end{proposition}

\begin{proof}
Let $E=E_{G,n}(\alpha )$ for brevity.
Let $G$ be a counterexample of minimal order, and let $M$ be a minimal $\alpha$-invariant normal subgroup of $G$. Since the image of $E$ is obviously equal to the similar subgroup constructed for $G/M$, we have $G=ME$ and $M\not\leq E$.
Then $M\cap E =1$ by minimality of $M$ because $M$ is abelian and $G=ME$.

Suppose that $C_M(\alpha )\ne 1$. Since $C_M(\alpha )$  normalizes $E$, we have $[C_M(\alpha ),E]\leq M\cap E=1$. Then $C_M(\alpha )$ is central in $G$ since $M$ is abelian and $G=EM$. By minimality, $C_M(\alpha )=M$ and $G = M \times E$. This
contradicts the hypothesis $G=[G,\alpha]$.

Thus,  $C_M(\alpha )=1$. Then the mapping $m\to [m,\alpha ]$ of $M$ into $M$ is injective, and therefore  surjective, since $M$ is finite. Hence every element of $M$ has the form $[m,\alpha ]$, and therefore also the form $[m,\alpha ,\alpha, \dots ,\alpha ]$ with $\alpha$ repeated $n$ times. Then $M\leq E$, which contradicts the equation $E\cap M =1$.
\end{proof}

\begin{proof}[Proof of Theorem~\ref{t-sol}] Recall that $g$ is an element of a finite soluble group $G$, and  the Fitting height of $E_n(g)$ is equal to $k$; we need to show that $g$ belongs to  $ F_{k+1}(G)$. Consider the subgroups
$$
[...[[G,g],g],\dots ,g]
$$
formed by taking successive commutator subgroups; in particular, these subgroups are subnormal in $G$. Let $H$ be the smallest of the subgroups $[...[[G,g],g],\dots ,g]$. Note that if $H=1$, then $g$ is a left-Engel element and therefore $g\in F(G)$ and there is nothing to prove.
In any case, $H=[H,g]$.

Let  $N=\langle H^G\rangle$  be the normal closure of $H$.
By construction, the image of $g$ in $G/N$ is a left-Engel element and therefore belongs to the Fitting subgroup of $G/N$. If we prove that the Fitting height of $N$ is at most $k$, then $N$, as a normal subgroup, is contained in $F_{k}(G)$ and then $g\in F_{k+1}(G)$, as required. Since $H$ is subnormal in $G$, the Fitting height
of its normal closure $N$ is the same as that  of $H$ by Lemma~\ref{l-sn}(b). Thus, it suffices to show that the Fitting height of $H$ is at most $k$.

Consider the automorphism $\alpha$ induced on $H$ by $g$ acting by conjugation. The subgroup $E_{H,n}(\alpha )$ is clearly contained in $E_n(g)$ and therefore has Fitting height at most $k$. Since $[H,\alpha ]=H$, by Proposition~\ref{pr-sol} we obtain $H=E_{H,n}(\alpha )$; hence the result.
\end{proof}

\section{Direct products of nonabelian finite simple groups}

We begin with elementary consequences of the Schreier conjecture, which may well be known.

\begin{lemma}\label{bou} If a finite group $G$ has a normal series all of whose factors are nonabelian simple groups, then $G$ is a direct product of nonabelian simple groups.
\end{lemma}

\begin{proof} Let
$$
1=G_0\lhd G_1 \lhd G_2\lhd \cdots \lhd G_{j-1}\lhd G_j=G
 $$
 be a normal series all of whose factors $G_i/G_{i-1}$ are nonabelian simple groups. We use induction on the length of the series $j$, the basis of which is obvious. For $j>1$, the quotient $G/G_{1}$ is a direct product of nonabelian simple groups by the induction hypothesis. The quotient $G/C_G(G_1)$ embeds in the automorphism group of $G_1$. Since the group of outer automorphisms of $G_1$ is soluble by
 the Schreier conjecture
 and $G$ has no soluble homomorphic images by the Jordan--H\"older theorem, we must  have     $G=G_1C_G(G_1)=G_1\times C_G(G_1)$, since $G_1\cap C_G(G_1)=1$. The result follows.
\end{proof}

We shall use without special references the well-known fact that in any direct product $S_1\times \dots \times S_r$ of nonabelian finite simple groups the only normal subgroups are products $S_{i_1}\times\dots\times S_{i_l}$ of some of the factors.

Recall that a subgroup $H$ of a direct product $G_1\times\dots\times G_r$ is called a subdirect product of the groups $G_1, \dots , G_r$ if the natural projection of $H$ onto every $G_i$ is equal to $G_i$.

\begin{lemma}\label{bbou} Let $G=S_1\times\dots\times S_r$ be a direct product of isomorphic nonabelian finite simple groups. Let $H\leq G$ be a subdirect product of the same groups.   Then $H=H_1\times\dots\times H_u$, where every $H_i$ is isomorphic to $S_1$.
\end{lemma}

\begin{proof}  Let $T_i$ denote the kernel of the $i$th projection of $H$.
Then intersections of some of the $T_i$ form a series as in Lemma \ref{bou}.
\end{proof}

For the rest of the section we shall work under the following setting.

\begin{hyp}\label{h-33}
Let $S=S_1\times\dots\times S_r$ be a direct product of $r$ isomorphic finite non-abelian simple groups and let $\varphi$ be the natural automorphism of $S$ of order $r$ that regularly permutes the $S_i$.
\end{hyp}

Note that $S$ has no proper  normal $\varphi$-invariant subgroups. In particular, $[S,\varphi ]=S$.

Let $D=C_S(\varphi )$ be the diagonal. More generally, a subgroup of $S$ that is the diagonal  (with respect to $\varphi $) in the product of some of the $S_i$ (not necessarily of all) is called a \textit{$d$-subgroup}. More precisely, a $d$-subgroup has the form
\begin{gather*}
\{(\dots , a_{i_1}, \dots , a_{i_2}, \dots , a_{i_l},\dots )\mid a_{i_s}=a_{i_1}^{\varphi ^{i_s-i_1}}\},
\end{gather*}
with segments of 1s (possibly empty) between the $a_{i_j}$.
In cases where we want to indicate the set of indices $I=\{i_1,\dots ,i_l\}$ of nontrivial coordinates in a $d$-subgroup, we  call it a $d(I)$-subgroup.

\begin{lemma}\label{l-d1} Under Hypothesis~\ref{h-33}, suppose that $K$ is the $d(I)$-subgroup for some $I$. Then $C_S(K)$ is the product of $S_i$ with $i\not\in I$ and $N_S(K)=K\times C_S(K)$.
\end{lemma}

\begin{proof}
This follows from the fact that $Z(S_i)=1$. The part about $C_S(K)$ is clear. If an element $b$ of $S$ has different coordinates $b_{u}\ne b_{v}$ for $u,v\in I$, then there is $x\in K$ with nontrivial coordinates $a$ such that $a^{b_{u}}\ne a^{b_{v}}$, so that $b\not\in N_S(K)$.
\end{proof}

\begin{lemma}\label{l-d2} Under Hypothesis~\ref{h-33}, suppose that $K$ is a $d$-subgroup and $x\in S$ is such that $K^x$ is also a $d$-subgroup. Then actually $x\in N_S(K)$.
\end{lemma}

\begin{proof}
This follows from the fact that  conjugation by $x\in S$ does  not change the set of indices of nontrivial coordinates.
\end{proof}

\begin{lemma}\label{l-d3} Under Hypothesis~\ref{h-33}, let a subgroup $H$ of $S$ be isomorphic to $S_1$ and suppose that $H$ is normalized by the diagonal $D$. Then $H$ is a $d$-subgroup.
\end{lemma}

\begin{proof} It is sufficient to show that if $H$ has nontrivial projections on some two subgroups $S_{i}$ and $S_{j}$, then the projection $P$ of $H$ onto $S_{i}\times S_{j}$ is equal to the $d(\{ i,j\} )$-subgroup $K$. If this is not the case, then $P\cap K=1$, since $P$ is normalized by $K$ by hypothesis and $K\cong P\cong S_1$. Then $PK=S_{i}\times S_{j}$ and therefore $P$ is normal in $S_{i}\times S_{j}$. Being a proper subgroup, it must be one of the factors, which contradicts the choice of $S_{i}, S_{j}$.
\end{proof}

\begin{lemma}\label{l-d4} Assume Hypothesis~\ref{h-33} with $r=|\varphi |$ being a prime.  If $H$ is a $\varphi $-invariant subgroup of $S$ containing the  diagonal $D$, then either $H=D$ or $H=S$.
\end{lemma}

\begin{proof}
Since $H\geq D$, by Lemma~\ref{bbou} we have $H=H_1\times \dots \times H_l$, where every $H_i$ is isomorphic to $S_1$. Since $D\leq H $ normalizes every $H_i$, by Lemma~\ref{l-d3} $H_i$ is a $d$-subgroup. We assume that $H\ne D$ and prove that $H=S$. Choose $H_i\ne D$. We claim that in fact $H_i=S_j$ for some $i,j$, which will imply that $H=S$ due to the $\varphi$-invariance.  Suppose that  $H_i$ has more than one nontrivial projections onto the $S_i$ --- but necessarily less than $|\varphi |$, since $H_i\ne D$ and $H_i$ is a $d$-subgroup. Since $r=|\varphi |$ is a prime, then we can choose a power of $\varphi $ such that $H_i^{\varphi ^w}\ne H_i$ and both $H_i^{\varphi ^w}$ and $H_i$ have nontrivial projections onto the same $S_k$. The former implies that $H_i^{\varphi ^w}=H_j\ne H_i$, so that $[H_i,H_i^{\varphi ^w}]=1$, while the latter implies that $[H_i,H_i^{\varphi ^w}]\ne 1$, a contradiction.
\end{proof}

\begin{lemma}\label{l2} Assume Hypothesis~\ref{h-33} with $r=|\varphi |$ being a prime. Let $n$ be a positive integer and let
$ F=\langle [x,\underbrace{\varphi ,\dots ,\varphi }_{n}]\mid x\in S_i\rangle$. Then $F=S$.
\end{lemma}

\begin{proof}
First note that  $[F,\varphi ]\ne 1$. Indeed, choose some element $x\in S_1$ of prime order $q\ne r$. Since $Q=\langle x ^{\langle \varphi \rangle}\rangle$ is an abelian $\varphi $-invariant $q$-group, we have
$$1\ne [Q,\varphi ]=[[Q,\varphi ],\varphi ]=\langle  [x,\underbrace{\varphi ,\varphi ,\dots ,\varphi} _{k}]\mid k\geq n+1\rangle .$$
In particular, $1\ne [x,\underbrace{\varphi ,\varphi ,\dots ,\varphi} _{n+1}]\in [F,\varphi ]$. It follows from the definition that $F$ is a $\varphi $-invariant subgroup, which is also normalized by $D$. Then $FD$ is a $\varphi $-invariant subgroup properly containing $D$ and therefore $FD=S$ by Lemma~\ref{l-d4}. As a result, $F$ is normal in~$S$. Being nontrivial and $\varphi $-invariant, it must coincide with $S$.
\end{proof}

\begin{lemma}\label{l2new} Assume Hypothesis~\ref{h-33}. Let $n$ be a positive integer and let
$ E=E_{S,n}(\varphi )=\langle [x,\underbrace{\varphi ,\dots ,\varphi }_{n}]\mid x\in S\rangle$. Then $E=S$.
\end{lemma}

\begin{proof}
Let $\varphi ^k$ be of prime order. Then $S=\langle [g,\underbrace{\varphi ^k,\dots ,\varphi ^k}_{n}]\mid g\in S_i,\; i=1,\dots ,r\rangle$ by Lemma~\ref{l2} applied to each orbit of $\varphi ^k$. It remains to observe that $[g,\underbrace{\varphi^k,\dots ,\varphi^k}_{n}]\in  E$ for $g\in S_i$. Indeed,   $A=\langle g ^{\langle \varphi \rangle}\rangle$ is an abelian $\varphi $-invariant group, in which repeated application of the formulae $[b,\varphi ^{i+1}]=[b^{\varphi},\varphi ^{i}][b,\varphi ]$ and $[bc, \varphi ^i]=[b, \varphi ^i][c, \varphi ^i]$ for $b,c\in A$ expresses the commutator $[g,\underbrace{\varphi^k,\dots ,\varphi^k}_{n}]$ in terms of commutators  $[x,\underbrace{\varphi ,\dots ,\varphi }_{n}]$ for $x\in A$.
\end{proof}

\section{Orbits in some transitive permutational actions}

We shall need several lemmas on transitive permutational actions of certain finite groups $G$ concerning the existence of exact (regular) orbits of an element $g\in G$. Rather than speaking about orbits, we prefer to state these lemmas in terms of intersections of $\langle g\rangle$ with a conjugate of the stabilizer of a point $H\leq G$. The following lemma is quite elementary.

\begin{lemma}\label{l-triv}
Let $g$ be an element, and $H$ a subgroup of a group $G$. Suppose that $g$ is contained in a subgroup $G_0\leq G$ and  $\langle g\rangle $ has trivial intersection with $(H\cap G_0)^x$ for some $x\in G_0$. Then $H^x\cap \langle g\rangle=1$.
\end{lemma}

\begin{proof}
Let $H_0=H\cap G_0$. Then
\begin{equation*}
H^x\cap \langle g\rangle=H^x\cap G_0\cap \langle g\rangle=(H \cap G_0)^x\cap \langle g\rangle=H_0^x\cap \langle g\rangle=1.\qedhere
\end{equation*}
\end{proof}

\begin{lemma}\label{pr4} Let $S=S_1\times\cdots\times S_l$ be a direct product of finitely many nonabelian finite simple groups. Let $\varphi $ be an automorphism of $S$ such that every $\varphi $-orbit of the permutational action on $\{S_1,\dots,S_l\}$ has $|\varphi |$ elements. Let $H$ be a subgroup of $S\langle \varphi \rangle$ such that $S\not\leq H$. Then there is $x\in S$ such that $H^x\cap \langle \varphi \rangle=1$.
\end{lemma}

\begin{proof} We proceed by induction on $|S\langle \varphi \rangle |$. As a basis of induction we can consider the case where $|\varphi |$ is a prime, where the result is obvious. Indeed,  if $\varphi \in \bigcap _{x\in S}H^x$, then, since $\bigcap _{x\in S}H^x$ is $S$-invariant,
$$
S=[S,\varphi ]\leq [S,\bigcap _{x\in S}H^x]\leq H,
 $$
  a contradiction with hypothesis.
Next, if the order of $\varphi $ is divisible by $p^2$ for some prime $p$, then we can work with $\varphi ^p$, which also satisfies the condition that all its orbits are of length $|\varphi ^p|$. By induction applied to $S\langle \varphi ^p\rangle$ and $H_0=H\cap S\langle \varphi ^p \rangle$ we obtain $x\in S$ such that $H_0^x\cap \langle \varphi ^p\rangle =1$,  whence $H^x\cap \langle \varphi ^p\rangle =1$ by Lemma~\ref{l-triv} and then $H^x\cap \langle \varphi \rangle =1$. So we can assume that the order of $\varphi $ is square-free and is divisible by at least two primes. Let $p$ be a prime divisor of $|\varphi |$. Let $\beta =\varphi ^p$ and let $\alpha$ be an element of order $p$ in $\langle \varphi \rangle$. By induction  there exists $x\in S$ such that $(H\cap S \langle \beta \rangle )^x\cap \langle \beta\rangle=1$, whence $H^x\cap \langle \beta\rangle=1$. If $\alpha \not\in H^x$, then the proof is complete. Therefore we can assume  that $\alpha\in H^x$. Replacing $H$ by $H^x$ we now assume that $H\cap \langle \beta \rangle=1$ and $\alpha\in H$.

We can choose a $\varphi $-orbit such that the product $T$  of the simple factors in this orbit satisfies $T\not\leq H$. If $T\ne S$, then by induction there is $x\in T$ such that $(H\cap T\langle \varphi \rangle)^x\cap  \langle \varphi \rangle=1$; then $H^x\cap \langle \varphi \rangle =1$ by Lemma~\ref{l-triv}. Thus, we can assume that $S$ is a product over one $\varphi $-orbit.  Let $C=C_S(\beta )$. Since $H^x\cap \langle \beta\rangle=1$ for all $x\in C$, we can assume that $\alpha ^x\in H$ for all $x\in C$. Since $\alpha ^x=\alpha [\alpha ,x]$ and $\alpha \in H$, we deduce that $[C,\alpha]\leq H$. Since $C$ is the direct product of $d$-subgroups corresponding to $\beta$-orbits, and these $d$-subgroups are regularly permuted by $\alpha$, it follows that $C=[C,\alpha]$. In particular, the diagonal $D$ of $S$ is contained in $H$ and so $H\cap S$ is a subdirect product of the $S_i$. Hence $H\cap S$ is a direct product of subgroups $H_i$ isomorphic to $S_1$  by Lemma~\ref{bbou}.

If  $SH\ne S\langle \varphi \rangle$, then $SH=S\langle \varphi ^t\rangle $ for some $t>1$. In this case we can apply the induction hypothesis to $S\langle \varphi ^t\rangle$, find  $x\in S$ such that $H^x\cap \langle \varphi ^t\rangle =1$, and then also $H^x\cap \langle \varphi  \rangle =  H^x\cap S\langle \varphi ^t\rangle \cap \langle \varphi  \rangle  =H^x\cap \langle \varphi ^t\rangle =1$, since $S\cap \langle \varphi \rangle =1$. Thus, we assume that $SH=S\langle \varphi \rangle$. Then $s\varphi =h$ for some $s\in S$ and $h\in H$. Note that $h$ induces  the same permutation on $\{S_1,\dots, S_l\}$ as $\varphi $.

Since $D\leq H $ normalizes every $H_i$, by Lemma~\ref{l-d3} we obtain that $H_i$ is a $d$-subgroup.
We claim that in fact $H_i=S_j$ for some $j$. Indeed, choose an element $a_1\in S_1$ such that $a_1\ne a_1^{-1}$. By assumption, $H^{a_1^{-1}}\cap \langle \varphi \rangle \ne 1$, so that there is a nontrivial element $\psi \in \langle \varphi \rangle$ such that $\psi ^{a_1}\in H$. Then for every $i$ we have $H_i^{\psi ^{a_1}}=H_j$ for some $j=j(i)$. Since $H_i^{\psi }$ is also a $d$-subgroup and $\big(H_i^{\psi }\big)^{[\psi ,a_1]} =H_i^{\psi ^{a_1}}=H_j$, by Lemma~\ref{l-d2} we obtain that $[\psi ,a_1]$ normalizes  $H_i^{\psi }  =H_j$. Thus, $\psi$ permutes the subgroups $H_u$. It follows that $H_j^{[\psi ,a_1]} =H_j$ for all $j$. Since $D\leq H$, there is a $d$-subgroup $H_j$ with nontrivial projections onto at least one of the factors $S_1$ or $S_k=S_1^{\psi}\ne S_1$, which contain the nontrivial projections $a_1$ and $(a_1^{-1})^{\psi}$ of $[\psi ,a_1]$. Since $a_1\ne a_1^{-1}$, it follows from Lemma~\ref{l-d1} that $H_j=S_1$ or $H_j=S_k$.

The latter implies that $S\leq H$ due to the invariance under $h\in H$, which induces the same permutation on the $S_u$ as $\varphi $. This contradicts the hypothesis.
\end{proof}

\begin{lemma}\label{pr5}
 Let $S=S_1\times\dots\times S_r$ be a direct product of $r$ isomorphic nonabelian finite simple groups, and let $\varphi $ be an automorphism of $S$ that transitively permutes the $S_i$. Suppose that for some prime $p$ the stabilizer of $S_1$ in $\langle\varphi\rangle$ is a
 $p$-subgroup, possibly trivial. Let $H$ be a subgroup of $S\langle \varphi \rangle$ such that $S\not\leq H$. Then there is $x\in S$ such that $H^x\cap \langle \varphi \rangle=1$.
\end{lemma}

Note that we do not exclude the case of $r=1$, when of course $\varphi $ is an automorphism of  $S=S_1$ of prime-power order.

\begin{remark}
  The condition in Lemma~\ref{pr5} that the stabilizer of $S_1$ in $\langle\varphi\rangle$ is a $p$-subgroup is essential. The smallest example is given by the alternating group $S_1={ Alt}_5$, its automorphism $\varphi$  induced by conjugation by an element of order $6$ in the ambient symmetric group ${ Sym}_5$, and $H$ being a stabilizer of a point in the natural representation of ${ Sym}_5$ on $5$ points.
\end{remark}

\begin{proof}[Proof of Lemma~\ref{pr5}]
 We proceed by induction on $|S\langle \varphi \rangle |$.
As a basis of induction we can consider the case where $|\varphi |$ is a prime, where the result is obvious. Indeed,  if $\varphi \in \bigcap _{x\in S}H^x$, then, since $\bigcap _{x\in S}H^x$ is $S$-invariant, $
S=[S,\varphi ]\leq [S,\bigcap _{x\in S}H^x]\leq H$,
  a contradiction with hypothesis.

Next, if the order of $\varphi $ is divisible by $q^2$ for some prime $q$, then we can work with $\varphi ^q$ in place of $\varphi $. There is an orbit (possibly one-element) of the permutation induced by $\varphi ^q$  on the set $\{S_1, \dots , S_r\}$ such that  the  product $T$ over this orbit satisfies $T\not\leq H$. Then the group $T\langle \varphi ^q\rangle$ satisfies the hypotheses  of the  lemma with $H_0=H\cap T\langle \varphi ^q\rangle$ in place of $H$. By induction we find $x\in T$ such that $H_0^x\cap \langle \varphi ^q\rangle =1$, whence $H_0^x\cap \langle \varphi \rangle =1$, and then
$H^x\cap \langle \varphi \rangle =1$ by Lemma~\ref{l-triv}. So we can assume that the order of $\varphi $ is square-free and is divisible by at least two primes.

If the Sylow $p$-subgroup of $\langle \varphi \rangle$ is trivial, then we get the result by Lemma~\ref{pr4}. So let $\alpha $ be an  element of order $p$ which generates the  stabilizer of $S_1$ in $\langle \varphi \rangle$, and let $\beta=\varphi ^p$. By induction there exists $x\in S$ such that $(H\cap S\langle\beta\rangle)^x\cap \langle \beta \rangle=1$, whence $H^x\cap \langle \beta \rangle=1$ by Lemma~\ref{l-triv}. If $\alpha\not\in H^x$, then the proof is complete; so we can assume  that $\alpha \in H^x$. Replacing $H$ by $H^x$ we now assume that $H\cap \langle \beta \rangle=1$ and $\alpha \in H$.

Let $D=C_S(\beta )$. Since $H^x\cap \langle \beta\rangle=1$ for all $x\in D$, we can assume that $\alpha ^x\in H$ for all $x\in D$. Since $\alpha ^x=\alpha [\alpha ,x]$ and $\alpha \in H$, we deduce that $D=[D,\alpha ]\leq H$. Hence $H\cap S$ is a subdirect product of the $S_i$. By Lemma~\ref{bbou} we obtain that $H\cap S$ is a direct product of subgroups $H_i$ isomorphic to $S_1$.

If  $SH\ne S\langle \varphi \rangle$, then $SH=S\langle \varphi ^t\rangle $ for some $t>1$.
We can choose an orbit of the permutation induced by  $\varphi ^t$ on the set $\{S_1, \dots , S_r\}$ such that the product $T$  of the simple factors in this orbit satisfies $T\not\leq H$. The group $T\langle \varphi^t\rangle$ satisfies the hypotheses  of the  lemma with $H_0=H\cap T\langle \varphi ^t\rangle$ in place of $H$. By induction we find $x\in T$ such that $H_0^x\cap \langle \varphi ^t\rangle =1$, whence
$H^x\cap \langle \varphi ^t\rangle =1$ by Lemma~\ref{l-triv}, and then $H^x\cap \langle \varphi \rangle =H^x\cap S \langle \varphi ^t\rangle \cap \langle \varphi \rangle =H^x\cap  \langle \varphi ^t\rangle=1$, since $S\cap\langle \varphi \rangle =1$.
Thus, we assume that $SH=S\langle \varphi \rangle$. Then $s\varphi =h$ for some $s\in S$  and $h \in H$. Note that $h$
induces  the same permutation on $\{S_1,\dots , S_r\}$ as $\varphi $.

Since $D\leq H $, it follows that $D$ normalizes every $H_i$, and by Lemma~\ref{l-d3} we obtain that every $H_i$ is a $d$-subgroup. If $D=H\cap S$, then $D$ is normalized by $s\varphi =h\in H$. Since $D$ is also normalized by $\varphi$, we obtain that $D$ is normalized by $s$, whence $s\in D$ by Lemma~\ref{l-d1}. Since $D\leq H$, we obtain that $\varphi \in H$, a contradiction with our assumption.
Therefore we can assume that there is $H_i\ne D$.

We claim that in fact $H_i=S_j$ for some $j$. Indeed, choose an element $a_1\in S_1$ such that $a_1\ne a_1^{-1}$. By assumption, $H^{a_1^{-1}}\cap \langle \varphi \rangle \ne 1$, so that there is a nontrivial element $\psi \in \langle \varphi \rangle$ such that $\psi ^{a_1}\in H$. If $\psi\not\in \langle \alpha\rangle$, then we argue in the same fashion as in the proof of Lemma~\ref{pr4}. Namely, $H_i^{\psi }$ is a $d$-subgroup for every $i$ and $\big(H_i^{\psi }\big)^{[\psi ,a_1]} =H_i^{\psi ^{a_1}}=H_j$; hence    $\big(H_i^{\psi }\big)^{[\psi ,a_1]} =H_i^{\psi }$ by Lemma~\ref{l-d2};  therefore $H_j^{[\psi ,a_1]} =H_j$ for all $j$. Since $D\leq H$, there is a $d$-subgroup $H_j$ with nontrivial projections onto at least one of the factors $S_1$ or $S_k=S_1^{\psi}\ne S_1$, which contain $a_1$ and  $(a_1^{-1})^{\psi}$. Since $a_1\ne a_1^{-1}$, it follows from Lemma~\ref{l-d1} that $H_j=S_1$ or $H_j=S_k$, which leads to $S\leq H$, a contradiction.

It remains to consider the case where $1\ne H^{a_1^{-1}}\cap \langle \varphi \rangle \leq \langle \alpha \rangle $ for all $a_1\in S_1$ such that $|a_1|\ne 2$. Since such elements $a_1$ generate $S_1$ and $\alpha$ is a nontrivial automorphism of $S_1$, there is $a_1\in S_1$ such that $a_1\ne a_1^{\alpha }$. Since $\alpha\in H$ and $\alpha ^{a_1}\in H$ by our assumption, we obtain a nontrivial element $[\alpha , a_1]\in S_1\cap H$. Since $H\cap S$ is a direct product of $d$-subgroups $H_i$, we must have $S_1=H_j$ for some $j$.

In any case, we have $H_i=S_j$ for some $i,j$. This implies that $H=S$ due to the invariance under $h\in H$, which induces the same permutation on the $S_u$ as~$\varphi $. This contradicts the hypothesis.
\end{proof}

Let $N$ be a normal subgroup of a group $G$, and $A$ a subgroup of $G/N$. We shall be saying for brevity that $B$ \textit{covers} $A$ if $BN/N\geq A$.

The following proposition is the main result of this section.

\begin{proposition}\label{prm}
Let $G\langle g\rangle$ be a finite group with a normal subgroup $G$ and a cyclic subgroup $\langle g\rangle$, and let $R$ be the soluble radical of $G$. Let $\langle g_0\rangle=C_{\langle g\rangle} (G/R)$,  so that the image $\bar g$ of $g$ in $\langle g\rangle/\langle g_0\rangle$ is the automorphism of $G/R$ induced by conjugation by $ g$.  Suppose that  the following conditions hold:
  \begin{itemize}
   \item[\textrm (1)]  $G$ is a minimal $g$-invariant subgroup covering $G/R$;
   \item[\textrm (2)] $G\cap \langle g\rangle\leq C_{\langle g\rangle}(G/R)$;
   \item[\textrm (3)]  $G/R=S_1\times\dots\times S_r$ is a direct product of isomorphic nonabelian finite simple groups, which are transitively permuted by $\bar g$;
   \item[\textrm (4)] for some prime $p$ the stabilizer of $S_1$ in $\langle\bar g\rangle$ is a  $p$-subgroup, possibly trivial.
 \end{itemize}
  Then for any subgroup $H\leq G\langle g\rangle$ such that $G\not\leq H$ there is $z\in G$ such that $H^{z}\cap \langle g\rangle \leq \langle g_0\rangle$.
\end{proposition}

 For brevity we shall refer to orbits of elements of $G\langle g\rangle$ on $\{S_i\}$ meaning orbits in the set $\{S_1, \dots , S_r\}$ of the corresponding induced permutations.

\begin{proof} Induction on $|G\langle g\rangle|$. As a basis of induction we can take the case $R=1$. In this case, $\langle g_0\rangle$ is a central subgroup, and in the quotient $G\langle g\rangle /\langle g_0\rangle$ the image of $H$ does not contain the image of $G$. Indeed, otherwise $G\leq H\langle g_0\rangle$ and then   $G=[G,G]\leq [H\langle g_0\rangle, H\langle g_0\rangle ]=[H,H]$, contrary to the assumption $G\not\leq H$. Thus, due to conditions (3), (4), we can apply Lemma~\ref{pr5} to the group $G\langle g\rangle /\langle g_0\rangle$ and  find  $z\in G$ such that  $H^z\cap \langle g\rangle \leq \langle g_0\rangle $, as required.

Now let $R\neq1$ and let $M$ be a minimal $g$-invariant abelian normal subgroup  of $G$. If $G\not\leq MH$, then we can apply induction to $G\langle g\rangle/M$. Indeed, since $M\leq R$, conditions (1), (3), (4) obviously hold for the images in  $G\langle g\rangle/M$, and the image of $\langle g_0\rangle$ is the centralizer of $G/R$ in the image of $\langle g\rangle$. We check condition (2): if $am=g^i$ for $a\in G$, $m\in M$, then $g^i\in G$ and therefore $g^i\in C_{\langle g\rangle }(G/R)$ by condition~(2) for $G\langle g\rangle$. By induction we find $z\in G$ such that  $H^z\cap \langle g\rangle \leq \langle g_0\rangle M$; then
$H^z\cap \langle g\rangle \leq \langle g_0\rangle M\cap \langle g\rangle \leq C_{\langle g\rangle}(G/R)= \langle g_0\rangle
$,
 as required.
Thus, we can assume that $G\leq MH$.

Suppose that $\langle g_0\rangle =C_{\langle g\rangle}(G/R)$ contains some nontrivial Sylow $q$-subgroup $\langle g_q\rangle $ of $\langle g\rangle $. Then $\langle g\rangle = \langle g_{1}\rangle \times \langle g_q\rangle $, where $\langle g_{1}\rangle$ is (necessarily nontrivial) Hall $q'$-subgroup. We choose a minimal $g_1$-invariant subgroup $G_1$ of $G$ covering $G/R$. Then the subgroups $G_1^{g_q^j}$ are of similar nature and together generate  $G$ by minimality of  $G$. Since $G\not \leq H$, we can assume without loss of generality that $G_1\not\leq H$.
The hypotheses of the proposition are satisfied for $G_1\langle g_{1}\rangle$ and $H\cap G_1\langle g_{1}\rangle$. Indeed, (1) is true by construction, (2) is inherited from the same condition for $g$, while (3) and (4) hold because the action of $g_1$ on $G_1/(G_1\cap R)$ is similar to the action of $g$ on $G/R$. Since $G_1\not\leq H\cap G_1\langle g_{1}\rangle$, by induction we find $z\in G_1$ such that $(H\cap G_1\langle g_{1}\rangle )^z\cap \langle g_{1}\rangle\leq \langle g_0\rangle$, whence  $H ^z\cap \langle g_{1}\rangle\leq \langle g_0\rangle$. Then, since $\langle g_q\rangle\leq \langle g_0\rangle$ by assumption,
 $$
 H ^z\cap \langle g\rangle =  \big(H ^z\cap \langle g_{1}\rangle \big)\times  \big(H ^z\cap \langle g_{q}\rangle\big)\leq \langle g_0\rangle.
 $$
 Thus, we can assume that every Sylow $q$-subgroup of $\langle g \rangle$ acts nontrivially on $G/R$.

Now suppose that  $GH \ne G\langle g\rangle $; then $GH\leq G\langle g^s\rangle \ne G\langle g\rangle$ for some prime $s$.  Note that $\langle g^s\rangle\geq \langle g_0\rangle$ by the assumption at the end of the last paragraph.
For a $g^s$-orbit on $\{S_i\}$, we choose a minimal $g^s$-invariant subgroup $G_1$ of $G$ covering the product of the $S_i$ in this $g^s$-orbit. Then the subgroups $G_1^{g^j}$ are of similar nature and  $G =\langle G_1,G_1^g,\dots ,G_1^{g^{s-1}}\rangle$ by minimality of   $G$. Since $G\not \leq H$, we can assume without loss of generality that $G_1\not\leq H$. Note that the centralizer of $G_1R/R$
in $\langle g^s\rangle $ is equal to $\langle g_0\rangle$ because the products of simple factors over $g^s$-orbits are permuted by $g$.
 The group $G_1\langle g^s\rangle$ satisfies the hypotheses  of the  proposition with $H_1=H\cap G_1\langle g^s\rangle$ in place of $H$. Indeed, conditions (1), (3) are satisfied by construction of $G_1$, and conditions (2), (4) are obviously inherited from the same conditions for $ \langle  g\rangle$.

By induction we find $z\in G_1$ such that $H_1^z\cap \langle g^s\rangle \leq \langle g_0\rangle$. Note that $G\langle g^s\rangle \cap \langle g \rangle=\langle g^s \rangle$, since $\langle g^s\rangle\geq \langle g_0\rangle$ and $G\cap \langle g\rangle\leq \langle g_0\rangle$ by condition (2). Then
\begin{align*}H^z\cap \langle g\rangle
&= H^z\cap G\langle g^s\rangle \cap \langle g\rangle
= H^z\cap \langle g^s\rangle  \\
&= H^z\cap G_1\langle g^s\rangle \cap \langle g^s\rangle
=(H\cap  G_1\langle g^s\rangle )^z\cap \langle g^s\rangle\\
&=H_1^z\cap \langle g^s\rangle \leq \langle g_0\rangle
\end{align*}
Thus, we can assume that $GH=G\langle g\rangle$.

It follows that $M\cap H=1$ by minimality of $M$. Indeed, this intersection is normal in $MH$ because $M$ is abelian, and $MH=G\langle g\rangle$ because $MH\geq GH=G\langle g\rangle$. We also have $M\not\leq H$, since $G\not\leq H$.

Let $a_1,\dots,a_s$ be elements of prime orders in $\langle g\rangle$, one for each prime divisor of $|g|$. Suppose that for some $i_0$ we have ${a^x_{i_0}}\not\in H$ for every $x\in G$, and let $|a_{i_0}|=q$. Let $\langle g_1\rangle $  be the Hall $q'$-subgroup of $\langle g\rangle$. We can assume that $g_1$ acts nontrivially on $G/R$, for otherwise $\langle g_1\rangle\leq \langle g_0\rangle$, and then  $H\cap \langle g\rangle\leq \langle g_0\rangle$ and we are done. The products of simple factors of $G/R$ over $g_1$-orbits are transitively permuted by $g$. Hence $g_1$ acts nontrivially on each such product, and the centralizer of such a product in $\langle g_1\rangle$ is equal to $\langle g_1\rangle\cap \langle g_0\rangle$. As above, since $G\not\leq H$, we can find a $g_1$-orbit on $\{S_i\}$ such that $H$ does not contain a minimal $g_1$-invariant subgroup $G_1$ of $G$ covering the product of the $S_i$ in this $g_1$-orbit.
Then
the hypotheses of the proposition hold for $G_1$, $g_1$, and $H_1=H\cap G_1\langle g_1\rangle$ in place of  $G$, $g$, and $H$. Indeed, conditions (1), (3) hold by construction, condition (4) is inherited from the same condition for $g$, and condition (2) follows  from the same condition for $g$ because an element in  $\langle g_1\rangle$ centralizing $G_1R/R$ also centralizes $G/R$. By induction we find $z\in G_1$ such that $H_1^{z}\cap \langle g_1\rangle\leq \langle g_0\rangle$, and therefore also
$$
H^{z}\cap \langle g_1\rangle =H^z\cap G_1\langle g_1\rangle\cap \langle g_1\rangle =(H\cap G_1\langle g_1\rangle )^z\cap \langle g_1\rangle =H_1^{z}\cap \langle g_1\rangle \leq \langle g_0\rangle.
 $$
 Since ${a^{z^{-1}}_{i_0}}\not\in H$ by our assumption, we obtain $H^{z}\cap \langle g\rangle\leq \langle g_0\rangle$. Thus, we can assume that for every $a_i$ there is $x\in G$ such that $a_i^x\in H$. Since $G\leq MH$, for every $i$ we can choose $m_i\in M$ such that $a_i^{m_i}\in H$.

We set
$$M_i=\{m\in M\mid {a^{m}_i}\in H\}.$$
Since $M$ is abelian, a direct calculation shows that $M_i=m_iC_M(a_i)$. Indeed,
if $a_i^m\in H$, then $a_i^m(a_i^{-1})^{m_i}\in M\cap H=1$, whence $a_i^{mm_i^{-1}}a_i^{-1}=1$, which means that $m\in m_iC_M(a_i)$. Thus, $M_i\subseteq m_iC_M(a_i)$. The reverse inclusion is obvious: $a_i^{mm_i}=a_i^{m_i}\in H$ for any $m\in C_M(a_i)$.

If there is  $z\in M\setminus \bigcup M_i$, then $a_i^z\not\in H$ for all $i$, so that $H^{z^{-1}}\cap \langle g\rangle =1$ and the proof is complete. Therefore we can assume that $M = \bigcup M_i$.

Suppose that $M=M_{i_0}$ for some $i_0$. Then $M=C_M(a_{i_0})$, which implies that $a_{i_0}\in H$ and, moreover, $a_{i_0}^x\in H$ for any $x\in  G\langle g\rangle $, since $G\leq MH$ and $g$ centralizes $a_{i_0}$. In other words,
$$
a_{i_0}\in K:=\bigcap _ {x\in G\langle g\rangle}H^x,
$$
where $K$ is a normal subgroup of $G\langle g\rangle$. It now follows that $a_{i_0}\in \langle g_0\rangle$: otherwise, $[G/R, a_{i_0}]=G/R$ and $[G,a_{i_0}]\leq [G,K]\leq K$, so that $K\cap G$ covers $G/R$ contrary to minimality of $G$, since $G\not\leq H\geq K$.

First suppose that $K_0=K\cap R\ne 1$. Then the hypotheses of the proposition hold for the images in $G\langle g\rangle /K_0$, which we denote by tilde. Indeed, $\tilde G/\tilde R\cong G/R$ and $C_{\langle \tilde g\rangle}(\tilde G/\tilde R)=\langle \tilde g_0\rangle$. Conditions (1), (3), (4) obviously hold since $K_0\leq R$. We check condition (2): if $ak=g^i$ for $a\in G$, $k\in K_0$, then $g^i\in \langle g_0\rangle$ by condition (2) for $G\langle  g\rangle$. Also, $\tilde G\not\leq \tilde H$ since $K_0\leq G\cap H$. By induction we find $\tilde z\in \tilde G$ such that $\tilde H^{\tilde z}\cap \langle \tilde g\rangle\leq \langle \tilde g_0\rangle$. Then for a pre-image $z\in G$ we have  $$H^{z}\cap \langle g\rangle\leq \langle  g_0\rangle K_0\cap \langle g\rangle\leq C_{\langle g\rangle}(G/R) = \langle g_0\rangle,$$
as required.

Thus we can assume that $K\cap R= 1$. Then $[a_{i_0}, G]\leq K\cap R=1$, so $a_{i_0}$ is central in $G\langle g\rangle$. We are going to apply induction to  $G\langle g\rangle /\langle a_{i_0}\rangle$, where  the images are denoted by tilde. Clearly,  $\tilde G/\tilde R\cong G/R$. We claim that  $C_{\langle \tilde g\rangle}(\tilde G/\tilde R)=\langle \tilde g_0\rangle$. Indeed, if $[G, g^i]\leq R    \langle a_{i_0}\rangle$, then  $[[G, g^i], g^i]\leq [R    \langle a_{i_0}\rangle,  g^i]\leq R $, whence $[G, g^i]\leq R  $, since then the image of $g^i$ in $G \langle g\rangle/ R$ belongs to the Fitting subgroup by Baer's theorem. We now verify the other hypotheses of the proposition. Condition (1): if $\tilde G_1\leq \tilde G$ for a $g$-invariant subgroup $\tilde G_1$ covering $\tilde G/\tilde R$, then for the full inverse image we have $G_1\leq G\langle a_{i_0}\rangle$. Then $[G_1,G_1]\leq [G\langle a_{i_0}\rangle, G\langle a_{i_0}\rangle]\leq G$. Since $[G_1,G_1]$ also  covers $G/R$, by minimality of $G$ we must have $[G_1,G_1]=G$, whence $\tilde G_1= \tilde G$. Condition (2): if $b=g^ia^j_{i_0}$ for $b\in G$, then $g^ia^j_{i_0}\in \langle g_0\rangle$ by condition (2) for $G\langle g\rangle$, whence $g^i\in \langle g_0\rangle$, since $a_{i_0}\in \langle g_0\rangle$. Conditions (3) and (4) obviously follow from the same conditions for $G \langle g\rangle$. Finally, $\tilde G\not\leq \tilde H$, since otherwise $G\leq H \langle a_{i_0}\rangle=H$, contrary to the hypothesis. By induction we find $z\in G$ such that $H^{z}\cap \langle g\rangle\leq \langle g_0\rangle \langle a_{i_0}\rangle =\langle g_0\rangle$.

Thus, we can assume that $M_i\ne M$ for every $i$.

Consider the subgroup $A=\langle {a^{m_1}_1},\dots , {a^{m_s}_s}\rangle\leq H$. Since $H\cap M=1$, we have $A\cong AM/M\cong \langle {a_1},\dots , {a_s}\rangle$, which is a cyclic group (a subgroup of $\langle g\rangle$). The subgroup $M$ is an elementary abelian $q$-group for some prime $q$. Let $B$ be a Hall $q'$-subgroup of  $A$ (possibly, $B=A$). In the semidirect product $MA=M\langle {a_1},\dots , {a_s}\rangle$, the Hall $q'$-subgroup of $\langle {a_1},\dots , {a_s}\rangle$, which is also a Hall $q'$-subgroup of $MA$, is  conjugate to  $B$ by an element $y\in M$.  Since $B\leq H$, this means that  for every $i$ such that $|a_i|\ne q$  we can replace all those elements $m_i$ by this element  $y\in m_iC_M(a_i)$. Without loss of generality, suppose that all the elements $a_2,\dots ,a_s$ are of order coprime to $q$, while $|a_1|$ may be equal to $q$, or not. If $|a_1|\ne q$, then
$$
M=yC_M(a_1)\cup yC_M(a_2)\cup \dots \cup yC_M(a_s),
$$
whence, after multiplying by $y^{-1}$,
\begin{equation}\label{e-m1}
M=C_M(a_1)\cup C_M(a_2)\cup \dots \cup C_M(a_s).
\end{equation}
If $|a_1|= q$, then
$$
M=m_1C_M(a_1)\cup yC_M(a_2)\cup \dots \cup yC_M(a_s),
$$
whence, after multiplying by $y^{-1}$,
\begin{equation}\label{e-m2}
M=y^{-1}m_1C_M(a_1)\cup C_M(a_2)\cup \dots \cup C_M(a_s).
\end{equation}
Suppose that $y^{-1}m_1C_M(a_1)\ne C_M(a_1)$. In this case, let $b=a_2\cdots a_s$. If  $b$ does not normalize the coset $y^{-1}m_1C_M(a_1)$, then \eqref{e-m2} implies
$$
(y^{-1}m_1C_M(a_1))^b\subseteq C_M(a_2)\cup \dots \cup C_M(a_s)
$$
and then also
$$
y^{-1}m_1C_M(a_1)\subseteq C_M(a_2)\cup \dots \cup C_M(a_s),
$$
since the right-hand side is $b$-invariant. Then
\begin{equation}\label{e-m3}
M=C_M(a_2)\cup \dots \cup C_M(a_s).
\end{equation}
If  $b$ normalizes the coset $y^{-1}m_1C_M(a_1)$, then, since the action of $b$ on $M$ is coprime, it follows that $y^{-1}m_1C_M(a_1)$ contains an element $m_0\in C_M(b)=\bigcap _{i=2}^sC_M(a_i)$. For any $x\in m_0C_M(a_1)=y^{-1}m_1C_M(a_1)\ne C_M(a_1)$ the element $m_0x$ is not in the coset $y^{-1}m_1C_M(a_1)$ and therefore $m_0x \in  C_M(a_i)$ for some $i\geq2$ by \eqref{e-m2}. Then $x\in m_0^{-1}  C_M(a_i)=C_M(a_i)$, so that again
$$
y^{-1}m_1C_M(a_1)\subseteq C_M(a_2)\cup \dots \cup C_M(a_s)
$$
and by \eqref{e-m2}
\begin{equation}\label{e-m4}
M=C_M(a_2)\cup \dots \cup C_M(a_s).
\end{equation}
As a result of \eqref{e-m1},\eqref{e-m3},\eqref{e-m4}, in all cases,
$$
M=C_M(a_1)\cup C_M(a_2)\cup \dots \cup C_M(a_s).
$$
Note also that the automorphism $a_1a_2\cdots a_s$ acts faithfully on $M$, as $M\ne C_M(a_i)$ for every~$i$. This situation is known to be impossible due to the following
well-known lemma, which we prove here for completeness.

\begin{lemma}
If $\alpha$  is an automorphism of an elementary abelian $q$-group $V$, then there is an element $v\in V$ such that $C_{\langle \alpha\rangle}(v)=1$.
\end{lemma}

\begin{proof}
Induction on $|V|$. If $|V|=q$, then every element in $\langle \alpha\rangle$ acts without non-trivial fixed points, and the result follows. In the general case, if an  element $\beta \in \langle \alpha\rangle$ of prime order $r\ne q$ has non-trivial fixed points, then $V=C_V(\beta )\times [V,\beta ]$ by Maschke's theorem, and both factors are $\alpha$-invariant and have smaller order than $V$. By induction, there are $v_1\in C_V(\beta )$ and $v_2\in [V,\beta ]$ such that $C_{\langle \alpha\rangle}(v_1)=C_{\langle \alpha\rangle}(C_V(\beta ))$ and $C_{\langle \alpha\rangle}(v_2)=C_{\langle \alpha\rangle}([V,\beta ])$. Then $C_{\langle \alpha\rangle}(v_1+v_2)= C_{\langle \alpha\rangle}(C_V(\beta )) \cap C_{\langle \alpha\rangle}([V,\beta ])=C_{\langle \alpha\rangle}(V)=1$.

Thus, we can assume that all elements of $\langle \alpha\rangle$ of order coprime to $q$ act on $V$ without nontrivial fixed points. It remains to choose an element of $V$ outside the centralizer of an element of $\langle \alpha\rangle$  of order $q$   (or any non-trivial element of $V$ if $q\nmid |\alpha |$).
\end{proof}

The proof of Proposition~\ref{prm} is complete.
\end{proof}

\section{Nonsoluble length}

In this section we prove Theorem~\ref{t3}. First we introduce some notation. Consider the `upper nonsoluble series' of $G$, which by definition starts from the soluble radical $R_0=R(G)$ and  the full inverse image $L_1$   of $F^*(G/R_0)$. Then by induction $R_j$ is the full inverse image of the soluble radical of $G/L_{j}$, and $L_{j+1}$ the full inverse image of $F^*(G/R_j)$. It is easy to see that the nonsoluble length $\lambda (G)$ is equal to the first positive integer $l$ such that $R_{l}=G$. In the normal series
\begin{equation}\label{e-riad}
1=L_0\leqslant R_0 <  L_1\leqslant R_1<  \dots \leqslant R_{l}=G
\end{equation}
each quotient $U_i=L_i/R_{i-1}$ is a (nontrivial) direct product of nonabelian simple groups, and each quotient $R_i/L_{i}$ is soluble (possibly trivial). By the well-known properties of the generalized Fitting subgroup,
if we write one of those  nonsoluble quotients as a direct product $U_j=S_1\times \dots \times S_v$ of nonabelian simple groups $S_i$, then the set of these factors $S_i$ is uniquely determined as the set of subnormal simple subgroups of $G/R_{j-1}$.
Acting by conjugation the group $G$ permutes these subnormal factors; for brevity we simply speak of orbits of elements of $G$ on $U_j$ meaning orbits in this permutational action. The stabilizer of a point $S_i$ can also be denoted as the normalizer $N_G(S_i)$ of the section $S_i$.

The subgroup $L_j/R_{j-1}$ contains its centralizer in $G/R_{j-1}$.  Let $K_j$ be the kernel of the permutational action of $G$ on $\{S_1, \dots , S_v\}$.  Clearly, $L_j\leq K_j$.  The quotient $K_j/L_j$  is soluble by the Schreier conjecture. Therefore, $K_j\leq R_{j}$. We shall routinely use these facts without special references.

Let $g$ be an element of a finite group $G$, and let $\{S_1,\dots, S_r\}$ be a $g$-orbit in $U_i$. We say that the orbit is \textit{pure} if the order of the automorphism of $S=S_1\cdots S_r$ induced by $g$ acting by conjugation is equal to $r$, which is the order of the permutation induced by $g$ on this orbit; in other words, if the stabilizer of a point in $\langle g\rangle$ acts trivially on $S$:  $N_{\langle g\rangle}(S_1)=C_{\langle g\rangle}(S)$.

We now prove two key technical propositions. It is convenient to introduce  the following hypothesis.

\begin{hyp}\label{h1} Let $g$ be an element of a finite group $G$,  and let $\{S_1,\dots, S_r\}$ be a $g$-orbit in a section $U_i=L_i/R_{i-1}$ of the series \eqref{e-riad} with $i\geq 2$. Suppose that $\langle g\rangle \cap L_i=1$,
    and  that  $g$ acting by conjugation induces a nontrivial automorphism $\bar g$ of $S=S_1\cdots S_r$. Let $t=|\bar g|$, so that $\langle  g^t \rangle$ is the centralizer of  $S$ in $\langle g \rangle$ and  $\langle \bar g \rangle =\langle g \rangle/\langle  g^t \rangle$.

Furthermore, let $\hat S$ be a minimal by inclusion $g$-invariant subgroup of $G$ such that $\hat SR_{i-1}/R_{i-1}=S$. If in addition $\{S_1,\dots, S_r\}$ is a pure $g$-orbit, then we choose  $\hat S$ to be also contained in $E_n(g)$, which is possible by Lemma~\ref{l2new}. Let $\hat S_1$ be a minimal by inclusion subgroup of $\hat S$ such that $\hat S_1R_{i-1}/R_{i-1}=S_1$.
\end{hyp}

The first of the technical propositions  provides a passage from a pure orbit of $g$ in $U_i$ to an orbit of at least the same length in the preceding section $U_{i-1}$.

\begin{proposition}\label{pr31} Assume Hypothesis~\ref{h1}.
   If $\{S_1,\dots, S_r\}$ is a pure $g$-orbit, then  $U_{i-1}$ contains a $g$-orbit $\{T_1, \dots , T_l\}$ of length divisible by $|\bar g|$ and for some choice of $\hat S_1$ there is an element $x_0\in \hat S_1\setminus R_{i-1}$
     that does not belong to $N_G(T_1)$.
     \end{proposition}

  \begin{proof}
Consider the permutational action of the group $\hat S\langle g \rangle$ on the set of simple factors of $U_{i-1} $. Since the kernel of this action is
contained in $R_{i-1}$,
 for an element $x_0\in \hat S_1\setminus R_{i-1}$ there is an orbit  $\{T_1,\dots , T_s\}$ such that the kernel $K$ of the restriction to this orbit does not contain $x_0$. In particular,  $\hat S\not\leq K$. Then $K$ cannot cover $S$ by minimality of $\hat S$. Since $K$ is a normal subgroup of $\hat S\langle g\rangle$, it follows that  $K\leq (\hat S\cap R_{i-1})\langle g^t\rangle$  (where, recall, $\langle g^t\rangle$ is the centralizer of $S$ in $\langle g\rangle$). In particular, $K\leq C_{\hat S\langle g\rangle}(S)$.

Let $W=\hat SK/K$
and let $R$ be the soluble radical  of $W$; then $W/R\cong S$ since $L_{j-1}\leq K$. Let $\tilde g$ be the image of $g$ in $\hat S\langle g\rangle/K$. We claim that  $C_{\langle \tilde g\rangle}(S)=\langle \tilde g^t\rangle$. Indeed, if $[\hat S, g^i]\leq (R_{i-1}\cap \hat S)K \leq  R_{i-1}   \langle g^t\rangle$, then  $[[\hat S, g^i], g^i]\leq R_{i-1} $, whence $[\hat S, g^i]\leq R_{i-1}  $, since then the image of $g^i$ in $\hat S \langle g\rangle/ (R_{i-1}\cap \hat S)$ belongs to the Fitting subgroup by Baer's theorem.

We consider the action of $W\langle \tilde g \rangle$ on the orbit $\{T_1,\dots , T_s\}$. Let $\tilde H$ be the stabilizer of $T_1$ in $W\langle \tilde g \rangle $.
Then $W\not \leq \tilde H$,
since $\hat S\not\leq K$.
We now show that the triple  $W$, $\tilde g$, $\tilde H$ satisfies the other hypotheses of Proposition~\ref{prm}.
 Condition (1): if $\tilde M\leq W $ for a $g$-invariant subgroup $\tilde M$ covering $S$, then for the full inverse image $M$ we have $M\leq \hat SK\leq \hat S\langle g^t\rangle$. Then $[M,M]\leq [\hat S\langle g^t\rangle, \hat S\langle g^t\rangle]\leq \hat S$. Since $[M,M]$ also covers $S$, by minimality of $\hat S$ we must have $[M,M]=\hat S$, whence $\tilde M=W$. Condition (2): if $ak=g^i$ for $a\in \hat S$, $k\in K$,  then $b=g^ig^{tj}$ for some $b\in \hat S$ and some $j$, because $K\leq (\hat S\cap R_{i-1})\langle g^t\rangle$. Since $\langle g\rangle\cap \hat S=1$ by Hypothesis~\ref{h1}, we obtain $b=g^ig^{tj}=1$, so that $g^i\in \langle g^t\rangle$ and $\tilde g^i\in C_{\langle \tilde g\rangle}(S)$. Conditions (3) and (4) clearly hold.

By Proposition~\ref{prm} there is $\tilde w\in W$ such that
$$
\tilde H^{\tilde w}\cap \langle \tilde g \rangle \leq \langle \tilde g^t \rangle.
$$
Let $H$ be the full inverse image of $\tilde H$ in $\hat S\langle g\rangle$, and $w$ some pre-image of $\tilde w$. Then
\begin{equation}\label{e-gw}
H^w\cap \langle g \rangle \leq \langle g^t \rangle K\cap \langle g \rangle \leq C_{\hat S\langle g\rangle}(S)\cap \langle g \rangle =\langle g ^t\rangle .
 \end{equation}
 Since $H$ is the stabilizer of $T_1$ in $\hat S\langle g\rangle$, this means that the $g$-orbit in $U_{i-1} $ containing $T_1^w$ has length divisible by $t=|\bar g|$.

Recall that $x_0\in \hat S_1\setminus R_{i-1}$ and $x_0\not\in K$.
There is $y\in  \hat S \langle g\rangle $ such that $x_0\not \in H^{wy}$. Since $y=g^js$ for some integer $j$ and some $s\in \hat S$, we have  $x_0\not \in H^{wy}=H^{wg^js}$, so that  $x_0^{s^{-1}}\not \in H^{wg^j}$. Thus the element $x_0^{s^{-1}}$, which  belongs to  $\hat S_1^{s^{-1}}\setminus R_{i-1}$, moves $T_1^{wg^j}$.  Note that $\hat S_1^{s^{-1}}$ is also a minimal subgroup of $\hat S$ covering $S_1$. Clearly, the  $g$-orbit containing $T_1^{wg^j}$ has the same length as the $g$-orbit
containing $T_1^{w}$, and this length divides  $|\bar g|$ by \eqref{e-gw}.
 It remains to rename $T_1^{wg^j}$
 by~$T_1$, \ $\hat S_1^{s^{-1}}$ by $\hat S_1$, and $x_0^{s^{-1}}$ by $x_0$. The proposition is proved.
  \end{proof}

  The next technical proposition  provides a passage from a non-pure orbit of $g$ in $U_i$ to an orbit in $U_{i-1}$ that is strictly greater with respect to the following ordering. Namely, on the set of positive integers we introduce the lexicographical order with respect to the exponents of primes in the canonical  prime-power decomposition: if
  $$
  a=2^{k_2}\cdot  3^{k_3}\cdot 5^{k_5}\cdot 7^{k_7} \cdots \qquad \text{and}\qquad
  b=2^{l_2}\cdot  3^{l_3}\cdot 5^{l_5}\cdot 7^{l_7} \cdots ,
    $$
  then by definition $a\prec b$ if, for some prime $p$, we have $k_q=l_q$ for all primes $q<p$ and $k_p<l_p$. Clearly, if $u$ is divisible by $v$, then $v\prec u$.

\begin{proposition}\label{pr32}  Assume Hypothesis~\ref{h1}.
   If $\{S_1,\dots, S_r\}$ is a non-pure $g$-orbit, then  $U_{i-1}$ contains a $g$-orbit $\{T_1, \dots , T_l\}$ of length strictly greater  than $r$ with respect to the order $\prec$.
\end{proposition}

\begin{proof} Let
$
|\bar g|=p_1^{\alpha _1}\cdot  p_2^{\alpha _2}\cdots $ for primes $p_1<p_2<\cdots $  and positive integers $\alpha _i$.
Let $\bar g_i$ be a generator of the Sylow $p_i$-subgroup of $\langle \bar g\rangle$, and let $ g_i$ be an inverse image of $\bar g_i$ in the Sylow $p_i$-subgroup of $\langle g\rangle$.

Since the orbit is non-pure, the stabilizer of a point $S_1$ in the permutational action of $\bar g$ on $\{S_1,\dots, S_r\}$ is nontrivial. Let the order of this stabilizer be
$p_1^{\beta  _1}\cdot  p_2^{\beta  _2}\cdots $
for the same primes  $p_i$ and non-negative  integers $\beta _i$, not all of which are zero. Let $j$ be the smallest index such that $\beta_j\geq 1$.
Consider the element
$$
f=g_1\cdots  g_{j-1}\cdot g_j
$$
(which generates the Hall $\{p_1,\dots , p_j\}$-subgroup of $\langle g\rangle$). Here, the case  ${j}=1$ is not excluded, when $f= g_1$.
Let $\bar f$ be the image of $f$ in $\langle g\rangle/\langle g^t\rangle$, so that $$|\bar f|= p_1^{\alpha _1}\cdots  p_j^{\alpha _j}.$$
 Let $t_0=|\bar f|$ for brevity.

Consider an orbit of $f$ in $\{S_1,\dots, S_r\}$, which we denote by $\{V_1,\dots , V_v\}$. Let $V=V_1\cdots  V_v$, and  let $\hat V$ be a minimal $f$-invariant subgroup of $\hat S$ such that $\hat VR_{i-1}/R_{i-1}=V$.
Let $\check f$ be the automorphism of $V$ induced by $f$.
Since the products of simple factors over $f$-orbits are permuted by $g$, the centralizer of $V$ in $\langle f\rangle $ is equal to  $\langle f\rangle\cap \langle g^t\rangle =\langle f^{t_0}\rangle$ (recall that $\langle g^t\rangle$ is the centralizer of $S$ in $\langle g\rangle$). Hence,
$$
|\check f|=|\bar f|=p_1^{\alpha _1}\cdots   p_{j-1}^{\alpha _{j-1}}\cdot p_{j}^{\alpha _j}.
$$
Note that by construction the stabilizer of a point $V_1$ in $\langle \check f\rangle $ is a $p_{j}$-subgroup.

  Consider the permutational action of the group $\hat V\langle f \rangle$ on the set of simple factors of $U_{i-1} $.
Since the kernel of this action is contained in $R_{i-1}$,
 we can choose an orbit $\{T_1,\dots , T_l\}$ for which the kernel $K$ of the restriction to this orbit does not contain some element $x\in \hat V_1\setminus R_{i-1}$. Then $K$ cannot cover $V$ by minimality of $\hat V$. Since $K$ is a normal subgroup of $\hat V\langle f\rangle$, it follows that  $K\leq (\hat V\cap R_{i-1})\langle f^{t_0}\rangle$. In particular,  $K\leq C_{\hat V\langle f\rangle}(V)$.

Let $W=\hat VK/K$ and let $R$ be the soluble radical  of $W$; then $W/R\cong V$, since $L_{j-1}\leq K$. Let $\tilde f$ be the image of $f$ in $\hat V\langle f\rangle/K$.  We claim that
$C_{\langle \tilde f\rangle}(V)=\langle \tilde f^{t_0}\rangle$. Indeed, if $[\hat V, f^i]\leq (R_{i-1}\cap \hat V)K \leq  R_{i-1}   \langle g^t\rangle$, then  $[[\hat V, f^i], f^i]\leq R_{i-1} $, whence $[\hat V, f^i]\leq R_{i-1}  $, since then the image of $f^i$ in $\hat V \langle f\rangle/ (R_{i-1}\cap \hat V)$ belongs to the Fitting subgroup by Baer's theorem.

We consider the action of $W\langle \tilde f \rangle$ on the orbit $\{T_1,\dots , T_l\}$. Let $\tilde H$ be the stabilizer of $T_1$ in $W\langle \tilde f \rangle $. Then $W\not\leq \tilde H$, since $\hat V\not\leq K$.

We now show that the triple  $W$, $\tilde f$, $\tilde H$ satisfies the other hypotheses of Proposition~\ref{prm}  in the role of $G$, $g$, $H$.
 Condition (1): if $\tilde M\leq W $ for an $f$-invariant subgroup $\tilde M$ covering $V$, then for the full inverse image $M$ we have $M\leq \hat VK\leq \hat S\langle f^{t_0}\rangle$. Then $[M,M]\leq [\hat V\langle f^{t_0}\rangle, \hat V\langle f^{t_0}\rangle]\leq \hat V$. Since $[M,M]$ also covers $V$, by minimality of $\hat V$ we must have $[M,M]=\hat V$, whence $\tilde M=W$. Condition (2): if $ak=f^i$ for $a\in \hat V$, $k\in K$,  then $b=f^if^{t_0j}$ for some $b\in \hat V$ and some $j$, because $K\leq (\hat V\cap R_{i-1})\langle f^{t_0}\rangle$. Since $\langle f\rangle\cap \hat V=1$ by Hypothesis~\ref{h1}, we obtain $b=f^if^{t_0j}=1$, so that $f^i\in \langle f^{t_0}\rangle$ and $\tilde f^i\in C_{\langle \tilde f\rangle}(V)$. Conditions (3) and (4) clearly hold by construction.

By Proposition~\ref{prm}
 there is $\tilde w\in W$ such that
$$
\tilde H^{\tilde w}\cap \langle \tilde f \rangle \leq \langle \tilde f^{t_0} \rangle.
$$
Let $H$ be the full inverse image of $\tilde H$ in $\hat V\langle f\rangle$, and $w$ some pre-image of $\tilde w$. Then
$$
H^w\cap \langle f \rangle \leq \langle f^{t_0} \rangle K\cap \langle f \rangle \leq C_{\hat V\langle f\rangle}(V)\cap \langle f \rangle =\langle f ^{t_0}\rangle .
 $$
 Since $H$ is the stabilizer of $T_1$ in $\hat V\langle f\rangle$,
this means that the $f$-orbit in $U_{i-1} $ containing $T_1^w$ has length divisible by
 $$
 |\check  f|=p_1^{\alpha _1}\cdots   p_{j-1}^{\alpha _{j-1}}\cdot p_{j}^{\alpha _j}.
$$
Then the $g$-orbit containing $T_1^w$ also has length divisible by
$$
p_1^{\alpha _1}\cdots   p_{j-1}^{\alpha _{j-1}}\cdot p_{j}^{\alpha _{j}}.
$$
This number is strictly greater with respect to our lexicographical order $\prec$ than the length of the original orbit
$$
p_1^{\alpha _1}\cdots   p_{j-1}^{\alpha _{j-1}}\cdot p_{j}^{\alpha _{j}-\beta _j}\cdots
$$
by the choice of $j$. It remains to rename $T_1^w$ by $T_1$. The proposition is proved.
\end{proof}

Theorem~\ref{t3}  will follow by induction from the following proposition. Recall that $K_i$ is the kernel
of the permutational action of $G$ on the set of simple factors of $U_i=L_i/R_{i-1}$.

\begin{proposition}\label{pr2}
Let $g$ be an element of a finite group $G$ whose order $|g|$ is equal to the product of $m$ primes counting multiplicities. Suppose that $\langle g\rangle \cap K_{ms}=1$ for some positive integer $s$. Then for any positive integer $n$ the nonsoluble length  of $E_n(g)$ is at least~$s$.
\end{proposition}

\begin{proof}
Consider the `upper nonsoluble series' for $E_n(g)$ constructed in the same way as \eqref{e-riad} was constructed for $G$, with its terms denoted by
\begin{equation}\label{e-rc}
1=\lambda _0\leqslant \rho _0< \lambda _1\leqslant \rho _1<  \dots \leqslant \rho _{e}=E_n(a),
\end{equation}
where $e=\lambda (E_n(g))$ is the nonsoluble length of $E_n(g)$.
Namely, $\rho_0=R(E_n(g))$ is the soluble radical,   $\lambda _1$ is  the full inverse image   of $F^*(E_n(g)/\rho_0)$, and by induction $\rho_j$ is the full inverse image of the soluble radical of $E_n(g)/\lambda _{j}$, and $\lambda _{j+1}$ the full inverse image of $F^*(E_n(g)/\rho_j)$. The quotients   $\lambda _i/\rho _{i-1}$ are (nontrivial) direct products of nonabelian simple groups, and the quotients $\rho _i/\lambda _{i}$ (possibly trivial) are soluble. Thus, $e$ is the first positive integer such that $\rho_{e}=E_n(g)$. Our task is to show that  $e\geq s$.

Since $\langle g\rangle \cap K_{ms}=1$, the element $g$ has at least one nontrivial orbit on the set of simple factors of $U_{ms}=L_{ms}/R_{ms-1}$, say, $\{S_1, \dots ,S_{r}\}$. If  the orbit is pure, then we apply Proposition~\ref{pr31}  to this orbit. If the orbit is not pure, then we apply Proposition~\ref{pr32}. Then we apply the same procedure to the orbit $\{T_1,\dots,T_l\}$ in $U_{ms-1}$ thus obtained: this orbit takes the role of the orbit $\{S_1, \dots ,S_{r}\}$ in Proposition~\ref{pr31} or \ref{pr32} depending on whether it is pure or not. We proceed with constructing this sequence of orbits, descending over the sections $U_i$ making $ms-1$ such steps. If we make such a step from a pure orbit by  Proposition~\ref{pr31}, then the length of the new orbit is divisible by the length of the old orbit and therefore does not decrease with respect to the order $\prec$. If we make such a step from a non-pure orbit by  Proposition~\ref{pr32}, then the length of the new orbit is strictly greater than the length of the old orbit with respect to the order $\prec$.

In the sequence of orbits thus constructed, some orbits may be pure, some not. We can visualize this fact as a sequence of $P$s (for pure) and $N$s (for non-pure), like
$$
P\to N\to P\to P\to P\to N\to N\to P\to N\to P\to
\cdots
$$
Importantly, in every passage of types $N\to N$ or $N\to P$ the length  of the orbit strictly increases with respect to the order $\prec$, while at passages of types $P\to N$ or $P\to P$ the length of the orbit does not decrease with respect to the order $\prec$. Therefore there can be at most $m-1$ passages of type $N\to N$ or $N\to P$, that is, at most $m-1$ occurrences of $N$ in this sequence. As a result, if the length of the sequence is at least $(s-1)m+(m-1)+1=ms$, then it will necessarily contain a subsequence
$$
P\to P\to  \cdots \to P
$$
of  $s$ pure orbits with $s-1$ consecutive passages $P\to P$. Let $i=t, t-1, \dots , t-s+1$ be the indices of the corresponding sections $U_i$.

Recall that if a $g$-orbit $\{S_1, \dots ,S_{r}\}$ in $U_i$ is pure, then by Lemma~\ref{l2new} the subgroup $S= S_1\times  \dots \times S_{r}$ is contained in the image of $E_n(g)$ in $G/R_{i-1}$ (since this image is obviously equal to the analogous subgroup $E_n(g)$ constructed for $G/R_{i-1}$).

The idea is to use each of these $s$ consecutive pure orbits  in $U_t, U_{t-1}, \dots , U_{t-s+1}$   to `mark' a nonsoluble factor of the series \eqref{e-rc} and prove that the factor `marked' by the pure orbit in $U_{i-1}$ is necessarily `lower' in \eqref{e-rc} than the factor marked by the pure orbit in $U_i$, for every $i=t, t-1,\dots  ,t-s+2$. Then the series \eqref{e-rc} must realize nonsoluble length at least $s$.

Thus, let $\{S_1, \dots ,S_{r}\}$ be a  pure $g$-orbit in $U_i$, and $\{T_1, \dots ,T_{l}\}$ the pure  $g$-orbit in $U_{i-1}$ obtained by Proposition~\ref{pr31}.
Recall that in  accordance with Hypothesis~\ref{h1}, $\hat S$ is a minimal by inclusion $g$-invariant subgroup of $E_n(g)$ such that $S=\hat SR_{i-1}/R_{i-1}$,  and $\hat S_1$ is a minimal by inclusion subgroup of $\hat S$ such that $\hat S_1R_{i-1}/R_{i-1}=S_1$. Note that since $S_1$ is nonabelian simple, $\hat S_1$ has no nontrivial soluble homomorphic images:
\begin{equation}\label{e-s1}
\hat S_1=[\hat S_1, \hat S_1].
\end{equation}
Recall also that by Proposition~\ref{pr31} we have an element $x_0\in \hat S_1\setminus R_{i-1}$
such that
\begin{equation}\label{e-11}
T_1^{x_0}\ne T_1.
\end{equation}

 Consider  the image of the series \eqref{e-rc} in $E_n(g)R_{i-1}/R_{i-1}$.  Since  $\hat S_1R_{i-1}/R_{i-1}\cong  S_1$ is a subnormal nonabelian simple  subgroup of   $E_n(g)R_{i-1}/R_{i-1}$, we obviously have a well-defined index $j$ such that
 \begin{equation}\label{e-mark1}
 \hat S_1\leq \lambda _jR_{i-1}\qquad\text{and}\qquad \hat S_1\not\leq \rho _{j-1}R_{i-1}.
 \end{equation}
  Note that then also
 \begin{equation} \label{e-lj}
 \hat S_1\leq \lambda _j(R_{i-1}\cap E_n(g)).
\end{equation}
It is also clear that the index $j$ depends only on $S_1$ (that is, it is independent of the choice of $\hat S_1$ such that $\hat S_1R_{i-1}/R_{i-1}=S_1$).

Since $\{T_1, \dots ,T_{l}\}$ is also a pure  $g$-orbit in $U_{i-1}$,  the  product $T= T_1 \cdots T_{l}$ is also covered by $E_n(g)$ by Lemma~\ref{l2new}. If $\hat T_1$ is  any subgroup of $E_n(g)$ such that $\hat T_1R_{i-2}/R_{i-2}=T_1$, then, again, there is a well-defined index $u$ depending only on $T_1$ such that
\begin{equation} \label{e2}
\hat T_1\leq \lambda _u(R_{i-2}\cap E_n(g))\qquad \text{and}\qquad
\hat T_1\not\leq \rho _{u-1}R_{i-2}.
\end{equation}

\begin{lemma}\label{l3}
Under the above hypotheses, $j>u$.
\end{lemma}

\begin{proof}
We argue by contradiction. Suppose that $j\leq u$. Then \eqref{e-lj} implies
that $\hat S_1\leq \lambda _u(R_{i-1}\cap E_n(g))$. The image of $\lambda _u/\rho _{u-1}$ in $E_n(g)/( R_{i-2}\cap E_n(g)$ is a direct product of nonabelian simple groups, one of which is $\bar T_1=\hat T_1/(\hat T_1\cap R_{i-2})\cong T_1$ by \eqref{e2}.  Acting by conjugation the group $E_n(g)$ permutes these factors. Consider the permutational action of $E_n(g)$ on the orbit containing $\bar T_1$. Clearly, $\lambda _u$ is contained in the kernel of this action. The subgroup $L_{i-1}\cap  E_n(g)$ normalizes $\hat T_1$ modulo $R_{i-2}\cap  E_n(g)$ (since $L_{i-1}$ normalizes $\hat T_1$ modulo $R_{i-2}$). As a normal subgroup of $E_n(g)$, then $L_{i-1}\cap  E_n(g)$ is contained in the kernel of that action on the orbit containing $\bar T_1$. Since $R_{i-1}/L_{i-1}$ is soluble, we obtain that the image of $\hat S_1\leq \lambda _u(R_{i-1}\cap E_n(g))$ in this action is soluble. But $\hat S_1 =[\hat S_1, \hat S_1]$  by the minimality of $\hat S_1$, as noted in \eqref{e-s1}. Therefore this image must actually be trivial; in particular,
\begin{equation}\label{e3}
\hat T_1^{x_0}\equiv \hat T_1 \,({\rm mod}\,\rho  _{u-1}(R_{i-2}\cap  E_n(g))).
\end{equation}

On the other hand, since $T_1^{x_0}\ne T_1$ by \eqref{e-11}, it follows that $[\hat T_1, \hat T_1^{x_0}]\leq R_{i-2}\cap E_n(g)$. Together with \eqref{e3} this implies
$$
\bar T_1=[\bar T_1,\bar T_1]\equiv 1  \,({\rm mod}\,\rho  _{u-1}(R_{i-2}\cap  E_n(g))),
$$
contrary to~\eqref{e2}.
\end{proof}

We now finish the proof of Proposition~\ref{pr2}. Each pure orbit $\{S_1,\dots ,S_r\}$ in $U_i$ for $i=t,\dots ,t-s+1$ in our sequence constructed by successive application  of Propositions~\ref{pr31} or \ref{pr32}  marks a nonsoluble quotient $\lambda _j/\rho _{j-1}$ of the series \eqref{e-rc} in the sense of \eqref{e-mark1}. By Lemma~\ref{l3} the next pure orbit in $U_{i-1}$ marks a strictly lower section. Therefore there must be at least $s$ different nonsoluble sections in \eqref{e-rc}, since their indices must be strictly descending as we go over the  $s$ consecutive pure orbits.
As a result, the nonsoluble length of $E_n(g)$ is at least $s$.
\end{proof}

\begin{proof}[Proof of Theorem~\ref{t3}] Recall that we have $\lambda (E_n(g))=k$ and $|g|$ is a product of $m$ primes counting multiplicities; we need to show that $g\in R_{(k+1)m(m+1)/2}(G)$. By Proposition~\ref{pr2} we have
$$
\langle g\rangle \cap R_{(k+1)m}\geq \langle g\rangle \cap K_{(k+1)m}\ne 1.
$$
The similar subgroup $E_n(\bar g)$ constructed for the image $\bar g$ of $g$ in $\bar G=G/R_{(k+1)m}$ is clearly the image of $E_n(g)$ and therefore its nonsoluble length is at most $k$. Then  $\bar g\in R_{(k+1)(m-1)m/2}(\bar G)$ by induction on $m$. The result follows, since $R_{(k+1)(m-1)m/2}(\bar G)$ is the image of $R_{(k+1)m(m+1)/2}$ in $G/R_{(k+1)m}$.
\end{proof}

\section{Generalized Fitting height}

\begin{proof}[Proof of Theorem~\ref{t2}]  Recall that $g$ is an element of a finite group $G$ whose order $|g|$ is equal to the product of $m$ primes counting multiplicities, and the generalized Fitting height of $E_n(g)$ is equal to $k$. We need to show $g$ belongs to  $ F^*_{((k+1)m(m+1)+2)(k+3)/2}(G)$.

As in the proof of Theorem~\ref{t-sol}, consider the subnormal $g$-invariant subgroups
$$
[...[[G,g],g],\dots ,g],
$$
 formed by taking successive commutator subgroups. Let $H$ be the smallest of the subgroups $[...[[G,g],g],\dots ,g]$. Note that if $H=1$, then $g$ is a left-Engel element and therefore $g\in F(G)$ and there is nothing to prove. In any case, $H=[H,g]$.

Let  $N=\langle H^G\rangle$  be the normal closure of $H$.
By construction, the image of $g$ in $G/N$ is a left-Engel element and therefore belongs to the Fitting subgroup of $G/N$. If we prove that the generalized Fitting height of $N$ is at most $h$ for some number $h$, then $N$, being  a normal subgroup, is contained in $F^*_{h}(G)$ and then $g\in F^*_{h+1}(G)$. Since $H$ is subnormal in $G$, by Lemma~\ref{l-sn}(b)
the generalized Fitting height
of its normal closure $N$ is the same as that  of $H$. Therefore it suffices to obtain the appropriate estimate of the generalized Fitting height of $H$.

Consider the group $H\langle g\rangle$. Since the nonsoluble length of $E_n(g)$ does not exceed its generalized Fitting height, by Theorem~\ref{t3} the element $g$ belongs to $R_{(k+1)m(m+1)/2}(G)$. Since $H=[H,g]$, it follows that $H\leq R_{(k+1)m(m+1)/2}(G)$, and since $H$ is a subnormal subgroup, $\lambda (H)\leq (k+1)m(m+1)/2$ by Lemma~\ref{l-sn}(a). It remains to obtain bounds for the Fitting height of the $(k+1)m(m+1)/2+1$ soluble factors of the upper nonsoluble series of $H$, which are $g$-invariant together with $H$.

Let $X=Y/Z$ be one of such soluble factors, where $Y$ and $Z$ are $g$-invariant and normal in $H$ and therefore subnormal in $G$.  Since $Y$ is subnormal, we have  $h^*(E_n(g)\cap Y)\leq h^*(E_n(g))=k$ by Lemma~\ref{l-sn}(a). Then the (soluble) image of $E_n(g)\cap Y$ in $Y/Z$ has Fitting height at most $k$. Since, clearly, $E_{Y,n}(g)\leq E_n(g)\cap Y$, we obtain that $E_{X,n}(g)$, which is the image of  $E_{Y,n}(g)$, also has Fitting height at most $k$.

Consider the `outer' semidirect product $X\rtimes \langle g\rangle$ (recall that both $Y$ and $Z$ are $g$-invariant, so this semidirect product is well defined). By applying Theorem~\ref{t-sol}
we obtain that $g\in F_{k+1}(X\rtimes \langle g\rangle )$. Therefore $[X, g]\leq F_{k+1}(X\rtimes \langle g\rangle )\cap X\leq F_{k+1}(X)$. In other words, $g$ acts trivially on $X/F_{k+1}(X)$. Since $H=[H,g]$, it follows that $H$ also acts trivially on $X/F_{k+1}(X)$, that is, $X/F_{k+1}(X)$ is a central section of $H$. In particular, $X=F_{k+2}(X)$, that is, the Fitting height of $X$ is at most $k+2$.

Thus, the generalized Fitting height of $H$ is at most
$((k+1)m(m+1)/2+1)(k+2)
+(k+1)m(m+1)/2=((k+1)m(m+1)+2)(k+3)/2 -1$, as required.
\end{proof}

\section{Final remarks and conjectures}
\label{s-rem}

We conjecture that the results of Theorems~\ref{t2} and \ref{t3} can be strengthened by removing the dependence on the order of the element. In fact, we have quite precise conjectures, with best-possible values. For generalized Fitting height we state the following.

\begin{conj}  \label{conj1}
Let $g$ be an element of a finite group $G$, and $n$ a positive integer. If the
generalized Fitting height of $E_n(g)$ is equal to $k$, then $g\in F^*_{k+1}(G)$.
\end{conj}

For nonsoluble length we state the following.
\begin{conj}  \label{conj2}
 Let $g$ be an element of a finite group $G$, and $n$ a positive integer. If the nonsoluble length of $E_n(g)$ is equal to $k$, then $g\in R_{k}(G)$.
     \end{conj}

As we show below, it is in this strongest form that Conjectures~\ref{conj1} and \ref{conj2} can be derived from an affirmative answer to the following question about automorphisms of direct products of nonabelian finite simple groups.

\begin{question}
\label{q1}
Let $S=S_1\times \dots \times S_r$ be a direct product of nonabelian finite simple groups, and $\varphi $  an automorphism of $S$ transitively permuting the factors.
Is it true that $E_{S,n}(\varphi )=S$?
\end{question}

Here, recall, $E_{S,n}(\varphi )=\langle [x,\underbrace{\varphi ,\dots ,\varphi }_{n}]\mid x\in S\rangle$,  where commutators are taken in the semidirect product $S\langle \varphi \rangle$. In Lemma~\ref{l2new} we obtained an affirmative answer to Question~\ref{q1} in the special case where the order of $\varphi$ is equal to the number of factors $r$, that is, when the stabilizer of a point in $\langle \varphi\rangle$ in the induced permutational action on $\{S_1,  \dots , S_r\}$ is trivial. However, in general Question~\ref{q1} seems rather difficult and remains open. The first step would be to consider the case of $\varphi$ of prime order, when of course the open question is about a single nonabelian finite simple group and its automorphism of prime order. A significant headway in this direction was recently made by Robert Guralnick (private communication).

The reduction of Conjectures~\ref{conj1} and \ref{conj2} to an affirmative answer to Question~\ref{q1}  can be conducted simultaneously based on the following proposition.

\begin{proposition}
\label{pr-red}
Let $\alpha$ be an automorphism
 of a finite group $G$ such that $G=[G,\alpha]$.
Assume that Question~\ref{q1} has an affirmative answer.  Then $E_{G,n}(\alpha )=G$.
\end{proposition}

\begin{proof}
Let $E=E_{G,n}(\alpha )$ for brevity.
Let $G$ be a counterexample of minimal order, and let $M$ be a minimal $\alpha$-invariant normal subgroup of $G$. Then $G=ME$ and $M\not\leq E$.

We need to consider two cases: $M$ can be an elementary abelian $q$-group for a prime $q$, or a direct product of isomorphic non-abelian simple groups.

If $M$ is an elementary abelian $q$-group, then the proof proceeds in exactly the same fashion as in the proof of Proposition~\ref{pr-sol} in the soluble case.
Then $E\cap M =1$ by minimality of $M$ because $M$ is abelian and $G=ME$, and so on.

Thus, let $M=S_1\times \dots \times S_r$, where the $S_i$ are isomorphic non-abelian simple groups. Suppose that $M\ne C_M(\alpha )$. By the affirmative answer to Question~\ref{q1} every product over an orbit of $\alpha $ in the permutational action on $\{S_1,  \dots,  S_r\} $ is contained in $E$ whenever $\alpha $ acts on this product nontrivially. Let $T$ be the product of those $S_i$ that are contained in $E$.  Then $T$ is normal in $E\langle \alpha\rangle $, since if $S_i\leq E$ and $x\in E\langle \alpha\rangle $, then $S_i^x\leq E$ and $S_i^x$ is again one of the $S_u$. Since $T$ is also normal in $M$, we obtain that $T$ is normal in $ME\langle \alpha\rangle=G\langle \alpha\rangle$. By minimality of $M$ then $M=T\leq E$, a contradiction. Thus, $M=C_M(\alpha)$. Since $C_G(\alpha ) $ normalizes $E$, then  $[E,M]\leq M\cap E$ and  $M\cap E$ is normal in $M$ and therefore in $ME\langle \alpha\rangle =G\langle \alpha\rangle$. Since $M\not\leq E$, by minimality of $M$ we have $M\cap E= 1$ and  therefore, $[E,M]=1$. As a result, $G = M \times E\langle \alpha\rangle$. This contradicts the hypothesis $G=[G,\alpha ]$.
\end{proof}

\begin{theorem}
Conjectures~\ref{conj1} and \ref{conj2} are true if Question~\ref{q1} has an affirmative answer.
\end{theorem}

\begin{proof}
As in the proofs of Theorems~\ref{t-sol} and \ref{t2},
consider the subnormal $g$-invariant subgroups
$$
[...[[G,g],g],\dots ,g],
$$
  and let $H$ be the smallest of these subgroups. Note that if $H=1$, then $g$ is a left-Engel element and therefore $g\in F(G)$ and there is nothing to prove. In any case, $H=[H,g]$.

Let  $N=\langle H^G\rangle$  be the normal closure of $H$.
By construction, the image of $g$ in $G/N$ is a left-Engel element and therefore belongs to the Fitting subgroup of $G/N$. If we prove that the generalized Fitting height of $N$ is at most $k$, then $N$, being  a normal subgroup, is contained in $F^*_{k}(G)$ and then $g\in F^*_{k+1}(G)$, as required. Similarly, if we prove that the nonsoluble length of $N$ is at most $k$, then as a normal subgroup $N$ is contained in $R_{k}(G)$ and then $g\in R_{k}(G)$, as required. Since $H$ is subnormal in $G$, by Lemma~\ref{l-sn}(b) the generalized Fitting height and the nonsoluble length
of its normal closure $N$ are the same as those  of $H$. Therefore it suffices to estimate these parameters of $H$. But $H=[H,g]\leq E_n(g)$ by Proposition~\ref{pr-red}.  Since $H$ is subnormal in $G$, the generalized Fitting height and the nonsoluble length of $H$ do not exceed the same parameters of $E_n(g)$.
 The result follows.
\end{proof}


\begin{thebibliography}{99}


\bibitem{ha-hi}  P. Hall and G. Higman, The $p$-length of a $p$-soluble group and reduction
theorems for Burnside's problem, \emph{Proc. London Math. Soc. (3)} {\bf
 6} (1956), 1--42.

\bibitem{hup} B.  Huppert, {\it Endliche Gruppen}. I,
Springer, Berlin, 1967.


\bibitem{khu-shu} {E. I. Khukhro} and  {P. Shumyatsky}, Words and pronilpotent subgroups in profinite groups, \emph{J.~Austral. Math. Soc.} {\bf  97}, no.~3 (2014), 343--364.

    \bibitem{wil83} J. Wilson, On the structure of compact torsion groups, {\it
Monatsh. Math.}, {\bf 96} (1983), 57--66.

\bibitem{zel90} E. I. Zelmanov, A solution of the Restricted
Burnside Problem for
groups of odd exponent, \emph{Izv. Akad. Nauk SSSR Ser. Mat.}, {\bf
54}, 42--59; English translation: \emph{Math. USSR Izvestiya},
{\bf
36} (1991),
41--60.

\bibitem{zel91} E. I. Zelmanov,
 A solution of the Restricted
Burnside Problem for
2-groups, \emph{Mat. Sbornik}, {\bf 182}, 568--592; English translation: \emph{Math. USSR Sbornik}, {\bf 72}
(1992), 543--565.

\bibitem{zel92} E. I. Zelmanov, On periodic compact groups, \emph{ Israel J. Math.} {\bf 77}, no.~1--2 (1992), 83--95.

\end{thebibliography}
\end{document}